\input ./style/arxiv-general.cfg
\documentclass[MSNbibl,number,citesort,dvips]{arxbj}
\makeatletter
   \@ifpackageloaded{graphicx}{}{\usepackage{graphicx}}
\makeatother

\volume{23}
\issue{1}
\pubyear{2017}
\firstpage{479}
\lastpage{502}
\doi{10.3150/15-BEJ750}
\docsubty{FLA}

\makeatletter
\newcommand{\eqref}[1]{(\ref{#1})}
\newcommand{\Prob}{\mathrm{P}}
\newproclaim{definition}{Definition}
\newtheorem{lemma}{Lemma}
\newtheorem{proposition}{Proposition}

\makeatother

\begin{document}
\begin{frontmatter}

\title{Exact confidence intervals and hypothesis tests for parameters
of discrete distributions}
\runtitle{Exact intervals and tests for discrete distributions}

\begin{aug}
\author[A]{\inits{M.}\fnms{M{\AA}ns}~\snm{Thulin}\corref{}\thanksref{A}\ead[label=e1]{mans.thulin@statistik.uu.se}}
\and
\author[B]{\inits{S.}\fnms{Silvelyn}~\snm{Zwanzig}\thanksref{B}\ead[label=e2]{zwanzig@math.uu.se}}
\address[A]{Department of Statistics, Uppsala University, 751 05 Uppsala, Sweden.\\ \printead{e1}}
\address[B]{Department of Mathematics, Uppsala University, 751 05 Uppsala, Sweden. \\ \printead{e2}}
\end{aug}

%
\received{\smonth{12} \syear{2014}}
%
\revised{\smonth{3} \syear{2015}}

\begin{abstract}
We study exact confidence intervals and two-sided hypothesis tests for
univariate parameters of stochastically increasing discrete
distributions, such as the binomial and Poisson distributions. It is
shown that several popular methods for constructing short intervals
lack strict nestedness, meaning that accepting a~lower confidence level
not always will lead to a shorter confidence interval. These intervals
correspond to a class of tests that are shown to assign differing
$p$-values to indistinguishable models. Finally, we show that among
strictly nested intervals, fiducial intervals, including the
Clopper--Pearson interval for a binomial proportion and the Garwood
interval for a Poisson mean, are optimal.
\end{abstract}

\end{frontmatter}

\section{Introduction}\label{introduction}
Hypothesis testing and interval estimation of parameters in discrete
distributions are two of the classic statistical problems, particularly
for the binomial and Poisson distributions, which remain two of the
most important statistical models. The fact that these distributions
are discrete makes it impossible to construct non-randomized confidence
intervals that have coverage equal to $1-\alpha$ for all values of the
unknown parameter $\theta$, and, equivalently, impossible to construct
two-sided tests with size equal to $\alpha$ for all pairs $(\alpha
,\theta_0)$, where $\theta_0$ denotes the value of $\theta$ under the
null hypothesis. It is however possible to construct confidence
intervals that have coverage at least equal to $1-\alpha$ for all
values of the unknown parameter, and tests that have size at most equal
to $\alpha$. Such intervals and tests are called exact, and are the
topic of this paper.


Given an observation $x$, the classic method of constructing exact
confidence intervals for parameters of some common discrete
distributions is to use the fiducial interval of Fisher \cite{fisher30,wang00}: $(\theta_L,\theta_U)$ where $\theta_L$ and $\theta
_U$ are such that
%
\begin{equation}
\label{fidint} \sum_{k\leq x}\Prob_{\theta
_L}(X=k)=
\alpha/2 \quad\mbox{and}\quad \sum_{k\geq x}\Prob_{\theta
_U}(X=k)=
\alpha/2.
\end{equation}
For the binomial parameter, the fiducial interval is known as the
Clopper--Pearson interval \cite{cp1} and for the mean of a Poisson
distribution it is known as the Garwood interval \cite{garwood36}.

The hypothesis $H_0: \theta=\theta_0$ can be tested against the
alternative $H_1: \theta\neq\theta_0$ by checking whether $\theta_0$
is contained in the fiducial interval. The $p$-value $\lambda_f(\theta
_0,x)$ of this test is two times the smaller $p$-value of two one-sided tests:
%
\begin{equation}
\label{fidtest}
\lambda_f(\theta_0,x)=\min \biggl(2\cdot
\sum_{k\leq x}\Prob_{\theta
_0}(X=k),2\cdot\sum
_{k\geq x}\Prob_{\theta_0}(X=k),1 \biggr).
\end{equation}

In their seminal paper on binomial confidence intervals, Brown \textit{et al.} \cite{bcd1}
write: ``The Clopper--Pearson interval is wastefully conservative and is
not a good choice for practical use, unless strict adherence to the
prescription $C(p, n) > 1-\alpha$ is demanded,'' where $C(p,n)$ denotes
the coverage probability. Instead they recommend using approximate
intervals, which obtain the nominal confidence level $1-\alpha$ in some
average sense, but have lower coverage for some values of $\theta$.
Such intervals are typically shorter than exact intervals, and their
corresponding tests typically have higher power. These advantages comes
at the cost that the actual confidence levels may be much lower than
stated and that the size of tests may be inflated. For popular
approximate intervals, the deviations in coverage from $1-\alpha$ may
be non-negligible even for large sample sizes \cite{thu2}. For this
reason, some statistician prefer to use exact methods like those
discussed in this paper, in order to guarantee that confidence levels
are not exaggerated and type I error rates are not understated.

When other criteria than merely coverage levels and expected lengths
are considered, exact confidence intervals can moreover compare
favourably to approximate intervals \cite{vh2,ne2}. Finally, even if
one prefers to use average coverage as a criterion for comparing
confidence intervals, it is of interest to study exact intervals due to
the facts that these intervals can be adjusted to have coverage
$1-\alpha$ on average, and that such adjusted intervals tend to have
shorter expected length than other approximate intervals \cite{re1,thu1}. For comparisons of exact and approximate intervals in the
binomial setting, and further arguments for using exact methods for
discrete distributions, see~\cite{thu2}.

Regarding the fiducial Clopper--Pearson interval, Brown \textit{et al.} \cite{bcd1} also
write ``better exact methods are available; see, for instance, \cite{bs1} and \cite{ca1}.'' Fiducial intervals are
equal-tailed, meaning that the lower bound is a $1-\alpha/2$ lower
confidence bound and that the upper bound is a $1-\alpha/2$ upper
confidence bound.
Several authors, including those mentioned by Brown \textit{et al.} \cite{bcd1} in the above
quote, have proposed shorter exact intervals that improve upon fiducial
intervals by letting the tail-coverages vary for different $x$, so that
their bounds no longer are $1-\alpha/2$ confidence bounds \cite{st3,st2,crow56,bs1,ca1,casella89,bl1,byrne01b,lurz,schilling14,wang14}.
Such intervals, known as strictly two-sided intervals, tend to have
less conservative coverage and are typically shorter than fiducial
intervals. Their use has been advocated by \cite{agresti1,re1,agresti2,hirji06,fay10a,fay10b,sommerville13} and \cite{lp14}, among others.

Unlike the equal-tailed fiducial intervals, the $p$-values of tests
corresponding to strictly two-sided confidence intervals can not be
written as two times the smaller $p$-value of two one-sided tests.
Instead, for some test statistic $T(\theta_0,X)$ satisfying mild
regularity conditions detailed in Section~\ref{nested}, the $p$-value of
a strictly two-sided test is defined as
\[
\lambda(\theta_0,x)=\Prob_{\theta_0}\bigl( T(
\theta_0,X)\geq T(\theta_0,x)\bigr).
\]
If the  null distribution of $T(\theta_0,X)$ is asymmetric, the
level $\alpha$ rejection region of such a test is not the intersection
of the rejection regions of two one-sided level $\alpha/2$ tests.

The main goal of this paper is to show that strictly two-sided
confidence intervals and hypothesis tests suffer from several problems.
These are illustrated in Figure~\ref{fig0}, in which the $p$-values and
interval bounds for the mean of a Poisson distribution are shown for
two tests and their corresponding confidence intervals. The first of
these is the strictly two-sided Sterne \cite{st3} interval, the other being
the fiducial Garwood \cite{garwood36} interval.

\begin{figure}

\includegraphics{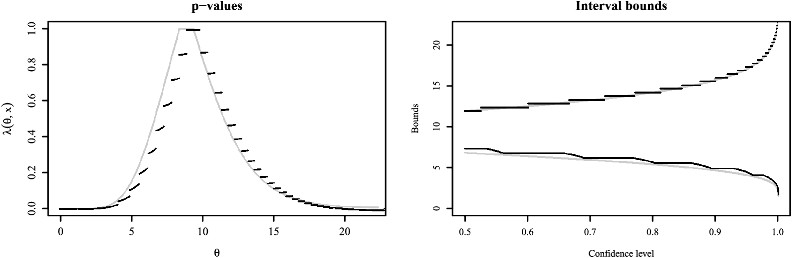}

\caption{$p$-values and interval bounds for the mean of a Poisson
distribution, when $x=9$ has been observed. The strictly two-sided
Sterne \cite{st3} method is shown in black, and the fiducial Garwood \cite{garwood36} method is shown in grey.}\label{fig0}
\end{figure}

In the spirit of Birnbaum \cite{birnbaum}, the $p$-values are plotted as a
function of the value $\theta_0$ of the parameter under the null
hypothesis. {In the Poisson model, it is reasonable to expect that a
small change in the null value of $\theta$ should lead to a small
change in the $p$-value, since $\Prob_\theta(X=x)$ is continuous in
$\theta$, so that there is no concernable difference between the
$\operatorname{Poisson}(\theta)$ and $\operatorname{Poisson}(\theta+\varepsilon)$ models when $\varepsilon$
is infinitesimal. This is not the case for the strictly two-sided test:
its $p$-value is discontinuous when viewed as a function of $\theta_0$.
The evidence against two models, which for all practical purposes are
indistinguishable, can therefore differ greatly. Several examples of
this are seen in Figure~\ref{fig0}; the $p$-value for $\theta
_0=4.954163$, for instance, is 0.0722, so that the null hypothesis is
rejected at the 10\% level, while the highly similar hypothesis $\theta
_0=4.954164$ cannot be rejected as its $p$-value is 0.1071.}

Moreover, we would expect that the $p$-value increases as $\theta_0$ goes
from 0 to the observed $x$, and that it thereafter decreases, since
this would mean that the $p$-value becomes smaller when the null
hypothesis agrees less with the data. This is not the case for the
strictly two-sided test. Instead, the $p$-value sometimes increases when
the null $\theta$ is changed to agree less with the observed $x$. {As
an example, consider the $p$-values shown in Figure~\ref{fig0}. When
$x=9$ has been observed from a Poisson distribution, the $p$-value when
$\theta_0=15.6$ is 0.0993, so that the null hypothesis is rejected at
the 10\% level. However, even though $x=9$ disagrees even more with
the null hypothesis $\theta_0=15.95$, the $p$-value for this $\theta_0$
is 0.1011, and the hypothesis can not be rejected.} The test
corresponding to the fiducial interval does not suffer from either of
these problems.



The strictly two-sided confidence interval is no better than its
corresponding test. When the interval bounds are plotted as functions
of the confidence level $1-\alpha$, we see two phenomenons. The first
is that the interval bounds are discontinuous in $1-\alpha$, meaning
that a small change in $\alpha$ can cause one of the interval bounds to
leap. The second is that the bounds sometimes are constant, meaning
that a change in $\alpha$ not necessarily will lead to a change in the
bounds. For some $\alpha$, both bounds remain unchanged in an interval
$(\alpha-\varepsilon,\alpha+\varepsilon)$. There is therefore no guarantee
that accepting a larger $\alpha$ will lead to a shorter interval; we
say that the interval is not \emph{strictly nested}. The fiducial
interval does not suffer from either of these problems.

{These properties can also cause strictly two-sided test and intervals
to behave strangely as more data is collected. As an example, consider
the Blaker  \cite{bl1} test for the negative binomial proportion $\theta$.
When $k=19$ successes are observed after $x=38$ trials, the maximum
likelihood estimator is $\hat{\theta}=0.5$ and Blaker $p$-value for the
test of the hypothesis $\theta=0.625$ is 0.0.0929, causing us to reject
the null hypothesis at the 10\% level. If we then decide to collect
more data by requiring that $k=20$ successes should be observed, and
observe one failure and one success so that $x=40$, $\hat{\theta}$ is
still $0.5$. We would now expect the $p$-value to decrease as this
outcome appears to be even less in line with $\theta=0.625$. Instead,
the Blaker $p$-value for $k=20$ and $x=40$ is 0.106, and we can no longer
reject the null hypothesis at the 10\% level. Analogous problems arise
for confidence intervals. The 90\% Blaker confidence interval for
$\theta$ given $k=19$ and $x=38$ is $(0.35992, 0.62279)$, while for
$k=20$ and $x=40$ it is $(0.36202, 0.62689)$. The latter interval is not,
as we normally would expect, a subset of the former. Moreover, the
interval based on more data is \emph{wider} than the interval based on
less data: the interval widths are 0.263 and 0.265, respectively.}

As we will see, intervals lacking strict nestedness is equivalent to
their corresponding $p$-values being discontinuous in $\theta$.
Consequently, intervals which are not strictly nested correspond to
tests that attach widely differing evidence to indistinguishable
hypotheses. {We believe that this is an unacceptable property of a
hypothesis test, and argue that such intervals and tests should be avoided.}

In this paper, we show that these problems are universal for strictly
two-sided intervals and tests, when the data is generated by a class of
discrete distributions that includes the binomial, Poisson and negative
binomial distributions. They also carry over to exact analysis of
contingency tables and discrete models with nuisance parameters, when
such analyses are based on conditioning that reduces the problem to a
one-parameter framework.

In Section~\ref{nested}, we give a formal description of the setting
for our results. We then show that the $p$-values of strictly two-sided
tests are discontinuous, and that their corresponding intervals have
bounds that are not strictly monotone. Finally, we show that strictly
two-sided intervals never are strictly nested, meaning that both
interval bounds simultaneously may remain unchanged when $\alpha$ is
changed. Section~\ref{strictly} is devoted to showing that strictly
two-sided intervals typically have bounds that moreover are
discontinuous in $\alpha$, and that the corresponding $p$-values lack
desirable monotonicity properties. In Section~\ref{optimality}, it is
then demonstrated that fiducial intervals not only are strictly nested
but also are the shortest equal-tailed intervals. The paper concludes
with a discussion in Section~\ref{discussion}. Most proofs and some
technical details are contained in two\vspace*{-3pt} appendices.

\section{The lack of strict nestedness and its implications}\label{nested}
\subsection{Setting}
This section is concerned with nestedness. We start by defining this concept.

\begin{definition}
A confidence interval is \emph{nested} if the $1-\alpha$ interval is a
subset of the $1-\alpha_0$ interval when $1>\alpha>\alpha_0>0$, and
\emph{strictly nested} if the $1-\alpha$ interval always is a proper
subset of the $1-\alpha_0$ interval.
\end{definition}

If an interval is not strictly nested, accepting a lower confidence
level does not always yield a shorter interval, so that sometimes
nothing is gained by increasing $\alpha$. Despite the importance of
nestedness, this property has not been discussed much in the
literature, likely because it is taken for granted. Notable exceptions
are Blaker \cite{bl1}, who proved that the binomial Blyth--Still--Casella
interval is not strictly nested and Vos and Hudson \cite{vh1}, who showed by example
that the Blaker interval for a binomial proportion lacks strict
nestedness. 

Next, we give some definitions and state the assumptions under which
strictly two-sided intervals are not strictly nested. We will limit our
study to parameters of discrete distributions $P_\theta$ belonging to a
class $\mathcal{P}(\Theta,\mathcal{X})$.

\begin{definition}\label{def2}
Let $\theta\in\Theta$ denote an unknown parameter, with $\Theta$ being
a connected open subset of~$\mathbb{R}$, and let $\mathcal{X}\subseteq
\mathbb{Z}$ be a sample space consisting of consecutive integers. A
family of distributions $P_\theta$ on $\mathcal{X}$ parameterized by
$\theta\in\Theta$ belongs to $\mathcal{P}(\Theta,\mathcal{X})$ if
\begin{longlist}[A3.]
\item[A1.] $\forall(\theta,x)\in\Theta\times\mathcal{X}$, $\Prob
_\theta(X=x)>0$,
\item[A2.] $P_\theta$ is stochastically increasing, i.e. $\Prob_\theta
(X\leq x)$ is strictly decreasing in $\theta$ for any fixed $x\in
\mathcal{X}\setminus\sup\mathcal{X}$,
%
\item[A3.] For any fixed $x\in\mathcal{X}$, $\Prob_\theta(X=x)$ is
differentiable in $\theta$.
\end{longlist}
\end{definition}

Conditions A1--A3 are satisfied by for instance the binomial, Poisson
and negative binomial distributions as long as $\Theta$ is the natural
parameter space, that is, as long as it has not been restricted. This
follows directly from the proposition below, the proof of which is
given in Appendix~\ref{proofs}. The conditions are typically also
satisfied for other common parameterizations.

\begin{proposition}\label{expfamprop}
If $P_\theta$ constitutes a regular discrete one-parameter exponential
family with an increasing likelihood ratio, where $\theta$ is the
natural parameter, then $P_\theta\in\mathcal{P}(\Theta,\mathcal{X})$.
\end{proposition}


To fully understand the implications of the lack of nestedness, we will
study the hypothesis tests to which non-nested intervals correspond,
so-called strictly two-sided tests:

\begin{definition}\label{def31}
Consider a two-sided test of $H_0: \theta=\theta_0$ versus $H_1: \theta
\neq\theta_0$, with a test statistic $T(\theta_0,x)$. The test is
called \emph{strictly two-sided} if the $p$-value of the test is $\lambda
(\theta_0,x)=\Prob_{\theta_0}( T(\theta_0,X)\geq T(\theta_0,x))$ and it
satisfies conditions \textup{B1}--\textup{B2} below. Moreover, in case $\lambda(\theta
,x)$, viewed as a function of $\theta$, has a jump at $\theta_0$ we
define $\lambda(\theta_0,x)=\liminf_{\theta\rightarrow\theta_0}\lambda
(\theta,x)$.
\begin{longlist}[B2.]
\item[B1.] For any $x\in\mathcal{X}$, there exists a $\theta_x\in\Theta
$ such that $T(\theta_x,x)< T(\theta_x,y)$ for all $y\in\mathcal
{X}\setminus{\{x\}}$.
%
\item[B2.] There exists a $\theta_0\in\Theta$ such that there does not
exist a $\mu\in\Theta$ for which $\Prob_{\theta_0}(T(\theta_0,X)=\mu
-k)=\Prob_{\theta_0}(T(\theta_0,X)=\mu+k)$ for all $k: \mu\pm k\in
\mathcal{X}$.
\end{longlist}
\end{definition}

Condition B1 is included to ensure that the test does not yield the
same result for all $x$ and $\theta$. The name strictly two-sided comes
from condition B2, which ensures that the $p$-value must be computed by
comparing the test statistic to \emph{both} tails of the null
distribution simultaneously.

The $p$-value of a strictly two-sided test can be written as
%
\begin{equation}
\label{Aset}\lambda(\theta,x)=\sum_{k\in\mathcal
{A}_{\theta,x}}
\Prob_\theta(X=k) \qquad\mbox{where } \mathcal{A}_{\theta,x}=\bigl\{k\in
\mathcal{X}:T(\theta,k)\geq T(\theta,x)\bigr\}.
\end{equation}

For simplicity, we will assume that the test statistic is such that
\begin{longlist}[B3.]
\item[B3.] For any $\theta\in\Theta$, there exists $ x_\theta\in\mathcal
{X}$ such that $T(\theta,x)$ is decreasing in $x$ when $x<x_\theta$ and
increasing in $x$ when $x>x_\theta$.
\end{longlist}
%
Under B3, the set $\mathcal{A}_{\theta,x}$ has a particularly simple form.

\begin{proposition}\label{B3prop}
Under \textup{B3}, the functions $k_1(\theta,x):=\min\{k\geq x_\theta: T(\theta
,k)\geq T(\theta,x) \}$ and $k_2(\theta,x):=\max\{k\leq x_\theta:
T(\theta,k)\geq T(\theta,x) \}$ are such that
%
\begin{equation}
\label{Aset2}\mathcal{A}_{\theta,x}=\bigl\{k\in\mathcal{X}: k \geq
k_1(\theta,x)\bigr\}\cup\bigl\{k\in\mathcal{X}: k \leq
k_2(\theta,x)\bigr\}.
\end{equation}
For any $x$, at least one of $k_1(\theta,x)$ and $k_2(\theta,x)$ is
non-constant in $\theta$.
\end{proposition}

The proof of the proposition is given in Appendix~\ref{proofs}. 

When $x$ is fixed and $\theta$ is varying we will refer to $\lambda
(\theta,x)$ as the $p$-value function. We define the corresponding
confidence interval using the convex hull of $\{ \theta: \lambda(\theta
,x)>\alpha\}$ to ensure that it in fact is an interval; as we will see
in Section~\ref{strictly}, $\{ \theta: \lambda(\theta,x)>\alpha\}$
itself is not always connected. The interval in the following
definition is guaranteed to be nested: if $\alpha>\alpha_0$ the convex
hull of $\{\theta: \lambda(\theta,x)>\alpha\}$ is a subset of the
convex hull of $\{\theta: \lambda(\theta,x)>\alpha_0\}$.

\begin{definition}\label{def3}
The $1-\alpha$ confidence interval $I_\alpha(x)=(L_\alpha(x),U_\alpha
(x))$ corresponding to a test is
%
\begin{equation}
\label{intdef}
I_\alpha(x)=\bigl(\inf\bigl\{ \theta: \lambda(\theta,x)>
\alpha\bigr\}, \sup\bigl\{ \theta: \lambda(\theta,x)>\alpha\bigr\}\bigr).
\end{equation}
A confidence interval is said to be \emph{strictly two-sided} if it is
based on the inversion of a strictly two-sided test.
\end{definition}

\subsection{Examples of strictly two-sided tests}
We will focus on four commonly used strictly two-sided tests, which
satisfy conditions B1, B2 and B3 for some common discrete
distributions, including the binomial, Poisson and negative binomial
distributions. These tests are briefly described below. Further
details, as well as conditions for B1--B3 to hold, are given in Appendix~\ref{app1}.

\emph{The likelihood ratio test}, for which $T(\theta,x)$ is the
likelihood ratio statistic \cite{hirji06,sommerville13}.

\emph{The score test}, for which $T(\theta,x)$ is the score statistic
\cite{hirji06,sommerville13}.

\emph{The Sterne test}, for which $T(\theta,x)=1/\Prob_\theta(X=x)$
\cite{st3}.

\emph{The Blaker test}, which in fact is a class of tests. Given a
statistic $S(x)$, the Blaker statistic is $T(\theta,x)=1/\min\{\Prob
_\theta(S(X)\leq S(x)), \Prob_\theta(S(X)\geq S(x))\}$, was introduced
in Blaker \cite{bl1}. See also \cite{xie} for a interpretation based on
confidence curves. In the binomial, negative binomial and Poisson
settings, we will use the sufficient statistic $S(x)=x$, as is common.

In Section~\ref{nonnest}, we will discuss confidence intervals that
have varying tail-coverage but are based on minimization algorithms
rather than test inversion. Because these intervals do not fall under
Definition~\ref{def3} we will refer to them as being of strictly
two-sided-type rather than as being strictly two-sided.

\subsection{Lack of strict nestedness and its interpretation}\label{lackofstrict}
We will now show that strictly two-sided intervals lack strict
nestedness, and that this is caused by jumps in the $p$-value function
$\lambda(\theta,x)$, viewed as a function of $\theta$. 

\begin{proposition}\label{thm3}
Assume that $P_\theta\in\mathcal{P}(\Theta,\mathcal{X})$. 
Let $\lambda(\theta,x)$ be the $p$-value function of a strictly two-sided
test and let $I_\alpha(x)$ denote its corresponding strictly two-sided
confidence interval. Then for any $x\in\mathcal{X}$:
\begin{longlist}
\item[(a)] $\lambda(\theta,x)$ is not continuous in $\theta$,
%
\item[(b)] the bounds of $I_\alpha(x)$ are not strictly monotone in
$\alpha$,
\item[(c)] $I_\alpha(x)$ is not strictly nested.
\end{longlist}
\end{proposition}

First, we show that $\lambda(\theta,x)$ has jumps. For any fixed $x\in
\mathcal{X}$, by Proposition~\ref{B3prop} we have, under B3,
%
\begin{equation}
\label{lambdaeq} \lambda(\theta,x)=\sum_{k\in\mathcal{A}_{\theta,x}}
\Prob_\theta (X=k)=\sum_{k\geq k_1(\theta,x) }
\Prob_\theta(X=k)+\sum_{k\leq
k_2(\theta,x) }
\Prob_\theta(X=k),
\end{equation}
%
where at least one of the $k_i(\theta,x)$ is non-constant in $\theta$.
$k_i(\theta,x)$ are integer-valued step-functions. Thus, for $\varepsilon
>0$ whenever $k_i(\theta,x)<k_i(\theta+\varepsilon,x)$, $k_i$ must have a
jump between $\theta$ and $\theta+\varepsilon$. This induces a jump in the
$p$-value function as well. To see this, assume without loss of
generality that $k_1(\theta+\varepsilon,x)=k_1(\theta,x)$ and $k_2(\theta
+\varepsilon,x)=k_2(\theta,x)+1$. Then
\[
\lambda(\theta+\varepsilon,x)=\sum_{k\geq k_1(\theta,x) }
\Prob_{\theta
+\varepsilon}(X=k)+\sum_{k\leq k_2(\theta,x) }
\Prob_{\theta+\varepsilon
}(X=k)+\Prob_{\theta+\varepsilon}\bigl(X=k_2(\theta,x)+1
\bigr),
\]
but by A1 and A3,
\[
\lim_{\varepsilon\searrow0}\lambda(\theta+\varepsilon,x)=\lambda(\theta ,x)+
\Prob_{\theta}\bigl(X=k_2(\theta,x)+1\bigr)>\lambda(\theta,x).
\]
Thus $\lambda(\theta+\varepsilon,x)\not\searrow\lambda(\theta,x)$ as
$\varepsilon\searrow0$ and the function is hence not continuous in $\theta
$. In particular, we have shown that $\lambda(\theta,x)$ has the
following property:

\begin{lemma}\label{B4}
Under the assumptions of Theorem~\ref{thm3}, $\lambda(\theta,x)$ as a
function of $\theta$ has a jump whenever a point is added to or removed
from $\mathcal{A}_{\theta,x}$.
\end{lemma}

Values of $\alpha$ for which $I_\alpha(x)$ is not strictly nested
correspond to the jumps in $\lambda(\theta,x)$. To see this, note that
if the interval $(\alpha_0,\alpha_1)\subseteq(0,1)$ is such that
%
\begin{equation}
\label{alpha0eq} \bigl\{\theta: \lambda(\theta,x)\in(\alpha_0,
\alpha_1)\bigr\}=\varnothing
\end{equation}
then for $\alpha\in(\alpha_0,\alpha_1)$, we have $\lambda(\theta
,x)>\alpha$ if and only if $\lambda(\theta,x)>\alpha_1$, which means
that the lower interval bound
\[
L_{\alpha}(x)=\inf\bigl\{\theta:\lambda(\theta,x)>\alpha\bigr\}=\inf\bigl
\{\theta :\lambda(\theta,x)>\alpha_1\bigr\}=L_{\alpha_1}(x)
\]
so that $L_{\alpha}(x)$ is not strictly monotone in $\alpha$. By
definition, the interval is not strictly nested if there exists an
$\alpha$ such that both $L_{\alpha}(x)$ and the upper interval bound
$U_{\alpha}(x)$ simultaneously are constant in a neighbourhood of
$\alpha$. The proof that there always exists such an $\alpha$ is
somewhat technical, and is deferred to Appendix~\ref{proofs}.


%
In particular, Proposition~\ref{thm3} holds when the test and its
corresponding confidence interval are exact. The proposition is
illustrated for exact tests and intervals in Figures~\ref{fig1}--\ref
{fig2}. In Figure~\ref{fig1}, $p$-values for the strictly two-sided \cite{st3,bl1},
likelihood ratio and score tests \cite{hirji06,sommerville13} are compared to the $p$-values of the
non-strictly two-sided test that corresponds to the fiducial interval
in the Poisson and binomial settings. It is readily verified that the
strictly two-sided tests satisfy B1--B3; see Appendix~\ref{app1}. In
Figure~\ref{fig2}, the interval bounds of some strictly two-sided
intervals are compared to the bounds of the fiducial interval. In the
Poisson case, the Sterne, Blaker, likelihood ratio, score,
Crow--Gardner \cite{crow56,casella89} and
Kabaila--Byrne  \cite{byrne01b} (the latter
two being of strictly two-sided-type) intervals are compared to the
Garwood interval. In the binomial case, the Sterne, Blaker, likelihood
ratio, score, Crow \cite{st2,bs1,ca1} (which is of strictly
two-sided-type) and G\"{o}b and  Lurz \cite{lurz} intervals are compared to the
Clopper--Pearson interval.

\begin{figure}

\includegraphics{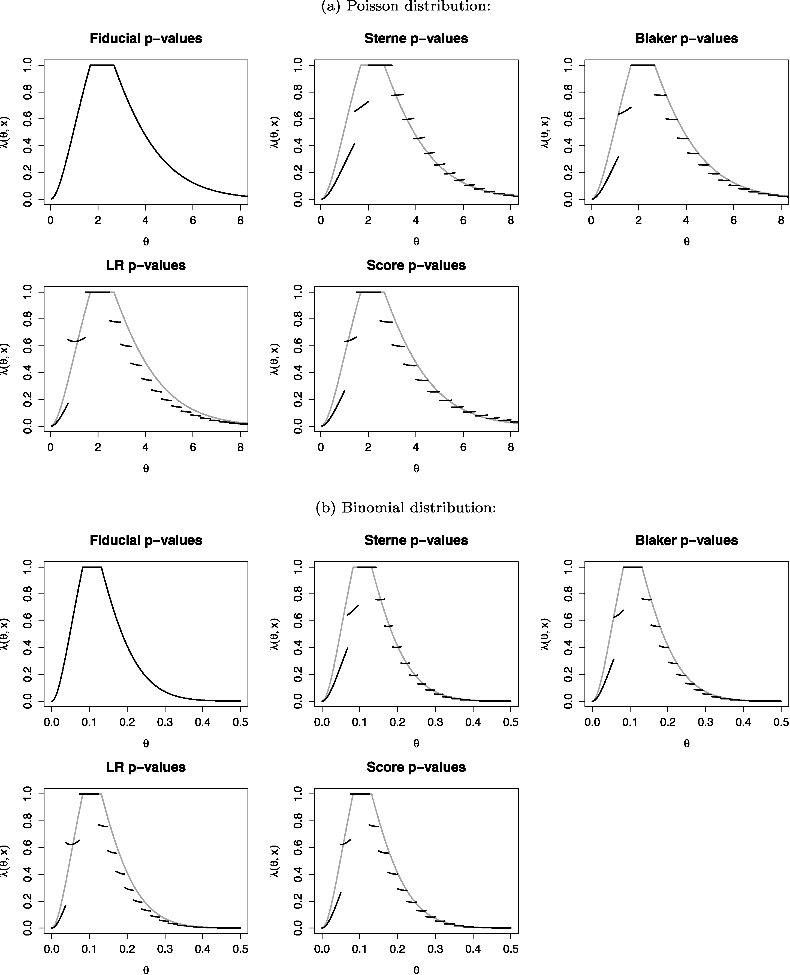}

\caption{Unlike the $p$-values for the fiducial test (shown in grey in
all plots), the strictly two-sided Sterne, Blaker, likelihood ratio
(LR) and score $p$-values are discontinuous and not bimonotone. In \textup{(a)},
the $p$-values are shown when $x=2$ is an observation from a Poisson
distribution with null mean $\theta$. In \textup{(b)}, the $p$-values are shown
when $x=2$ is an observation from a null $\operatorname{Bin}(20,\theta)$-distribution.}\label{fig1}
\end{figure}
\begin{figure}

\includegraphics{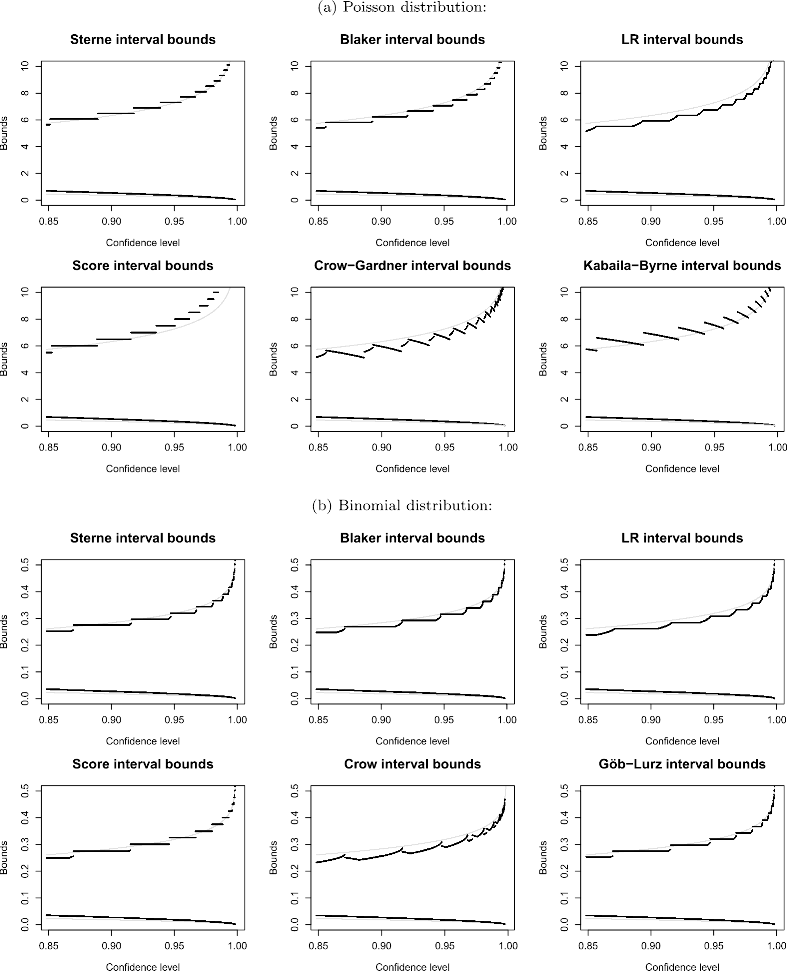}

\caption{Interval bounds of several strictly two-sided and strictly
two-sided-type confidence intervals. The intervals are compared to the
fiducial interval, the bounds of which are plotted in grey. In \textup{(a)}, the
intervals are shown when $x=2$ is an observation from a Poisson
distribution with mean $\theta$. In \textup{(b)}, the intervals are shown when
$x=2$ is an observation from a $\operatorname{Bin}(20,\theta)$-distribution.}\label{fig2}
\end{figure}

\subsection{The largest \texorpdfstring{$\alpha$}{alpha} for which an interval is strictly
nested}\label{smallest}
Proposition~\ref{thm3} tells us that strictly two-sided confidence
intervals lack strict nestedness and that their bounds are not strictly
monotone in $\alpha$. This may however not be a great problem if the
lack of strict nestedness and monotonicity occurs only for $\alpha$
close to 1.

Under some stronger assumptions on $T(\theta,x)$, $\mathcal{X}$ and
$P_\theta$ we can derive expressions for the largest $\alpha$ for which
$I_\alpha(x)$ is strictly nested and the largest $\alpha$ for which
each interval bound is strictly monotone. As we will see, these bounds
for $\alpha$ are usually close to 0, meaning that the lack of strict
nestedness and monotonicity occurs also for $\alpha$ that are used in practice.

We restrict our attention to samples spaces of the form $\mathcal{X}=\{
0,1,2,\ldots\}$ or $\mathcal{X}=\{0,1,2,\ldots, n\}$, for some known
$n<\infty$. Moreover, we will require some additional conditions, which
essentially make up stronger versions of A2 and B3:
\begin{enumerate}[A2$^+$.]
\item[A2$^+$.] $\Prob_\theta(X\leq x)$ is strictly decreasing in $\theta
$ for any $x\in\mathcal{X}\setminus\sup\mathcal{X}$.
\item[B3$^+$.]
\begin{enumerate}[(iii)]
\item[(i)] For any $\theta\in\Theta$, there exists $ x_\theta\in\mathcal
{X}$ such that $T(\theta,x)$ is strictly decreasing in $x$ when
$x<x_\theta$ and strictly increasing in $x$ when $x>x_\theta$.
\item[(ii)] For any $x\in\mathcal{X}$, there exists a $\theta_x\in\Theta
$ such that $\lambda(\theta_x,x)=1$ and $T(\theta,x)$ is strictly
decreasing in $\theta$ when $\theta<\theta_x$ and strictly increasing
in $\theta$ when $\theta>\theta_x$.
\item[(iii)] $x_\theta$ is an increasing function of $\theta$.
\end{enumerate}
\end{enumerate}

\begin{proposition}\label{smallprop}
Assume that $\mathcal{X}=\{0,1,2,\ldots\}$ or $\mathcal{X}=\{
0,1,2,\ldots,n\}$. Under \textup{A2}$^+$, \textup{B3}$^+$ and the assumptions of
Proposition~\ref{thm3} it holds that
\begin{longlist}[(a)]
\item[(a)] There exists an $\alpha_{\mathrm{nest}}>0$ such that $I_\alpha(x)$ is
strictly nested for all $x\in\mathcal{X}$ and $\alpha\leq\alpha_{\mathrm{nest}}$.
\item[(b)] Let $\alpha_L=\inf_{x\in\mathcal{X}}\inf_{\theta\in\{\theta:
T(\theta,0)>T(\theta,x)\}}\lambda(\theta,x)$. Then \textup{(i)} $\alpha_L>0$,
\textup{(ii)} for all $x>0$, $L_\alpha(x)$ is continuous and strictly increasing
in $\alpha$ when $\alpha\leq\alpha_L$, and \textup{(iii)} there exists an $x>0$
and an $\varepsilon>0$ such that $L_\alpha(x)$ is constant in $(\alpha
_L,\alpha_L+\varepsilon)$.
\item[(c)] For $\mathcal{X}=\{0,1,2,\ldots,n\}$, let $\alpha_U=\inf_{x\in\mathcal{X}}\sup_{\theta\in\{\theta: T(\theta,n)>T(\theta,x)\}
}\lambda(\theta,x)$. Then \textup{(i)} $\alpha_U>0$, \textup{(ii)} for all $x<n$,
$U_\alpha(x)$ is continuous and strictly decreasing in $\alpha$ when
$\alpha\leq\alpha_U$, and \textup{(iii)} there exists an $x>0$ and an $\varepsilon
>0$ such that $U_\alpha(x)$ is constant in $(\alpha_U,\alpha_U+\varepsilon)$.
\end{longlist}
\end{proposition}
%

\begin{figure}

\includegraphics{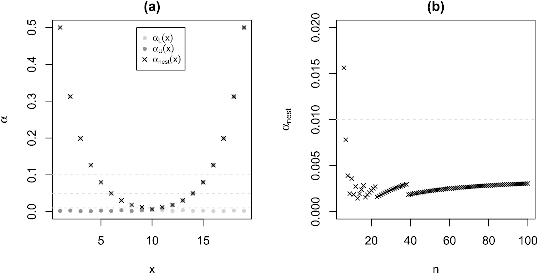}

\caption{\textup{(a)} The largest $\alpha$ for which the lower and upper bounds
of the Blaker interval for a binomial proportion are strictly monotone
($\alpha_L(x)$ and $\alpha_U(x)$), and the largest $\alpha$ for which
the interval is nested conditioned on $x$ ($\alpha_{\mathrm{nest}}(x)$), when
$n=20$. The common choices $\alpha\in\{0.01,0.05,0.1\}$ are shown as
dashed lines. \textup{(b)} $\alpha_{\mathrm{nest}}$, the largest $\alpha$ for which the
Blaker interval for a binomial proportion is strictly nested, as a~function of $n$.}\label{fig3}
\end{figure}

Proposition~\ref{smallprop} deals with $\alpha$ guaranteeing strict
monotonicity and nestedness for all $x$. We can also study monotonicity
and nestedness for fixed $x$. For any $x\in\mathcal{X}$, let $\alpha
_L(x)$ denote the largest $\alpha$ for which $L_\alpha(x)$ is strictly
monotone, and $\alpha_U(x)$ denote the largest $\alpha$ for which
$U_\alpha(x)$ is strictly monotone. Finally, let $\alpha_{\mathrm{nest}}(x)$ be
the largest $\alpha$ for which $I_\alpha(x)$ is strictly nested. In
Figure~\ref{fig3}(a), these quantities are shown for the Blaker
interval for a binomial proportion, with $n=20$ and $x\in\{1,2,\ldots
,19\}$. In this example, $\alpha_{\mathrm{nest}}(x)<0.1$ for most $x$. As is
seen, $\alpha_{\mathrm{nest}}(x)$ is often equal to or very close to $\max(\alpha
_L(x),\alpha_U(x))$. Figures for other intervals, other $n$ and other
distributions are similar.

Figure~\ref{fig3}(b) shows $\alpha_{\mathrm{nest}}$ for the binomial Blaker
interval as a function of the sample size $n$. It is seen that when
$7\leq n\leq100$ we have $\alpha_{\mathrm{nest}}<0.01$ for the Blaker interval,
meaning that the interval lacks strict nestedness for virtually all
values of $\alpha$ that actually are used in practice for these sample sizes.

\subsection{Confidence intervals not based on test-inversion}\label{nonnest}
An interesting class of confidence intervals are based on minimization
algorithms. This class includes \cite{st2,crow56,bs1,ca1,casella89,byrne01b} and \cite{schilling14} intervals. For such intervals, the shortest interval is
determined for each $\alpha$. What typically occurs for these intervals
is that they correspond to inversion of different tests for different
$\alpha$. Often this will result in intervals that lack nestedness (and
not only strict nestedness), as it leads to some values of $\theta$
having multiple $p$-values attached to them. This can be seen in Figure~\ref{fig2}: neither the Crow interval for the binomial parameter nor
the Crow--Gardner and Kabaila--Byrne intervals for the Poisson parameter
are nested.


If a two-sided $1-\alpha$ interval is $(\theta_\ell,\theta_u)$, then
the $p$-values for the corresponding two-sided tests of the hypotheses
$\theta_0=\theta_\ell$ and $\theta_0=\theta_u$ are $\alpha$. Using this
relationship, we can plot the $p$-value functions of tests corresponding
to intervals that are not defined in terms of test inversion, such as
minimization-based intervals. The lack of nestedness means that the
$p$-value function $\lambda(\theta,x)$ of the corresponding test is not a
proper function for $x\in\mathcal{X}$ fixed, since some values of
$\theta$ are mapped to more than one $p$-value. For some intervals, this
problem becomes extreme. Two examples of this are the Kabaila--Byrne
and Crow--Gardner intervals for a Poisson mean, shown in Figure~\ref{fig4}. For other intervals, the lack of nestedness results in less
extreme $p$-value functions. An example of this is the Schilling--Doi
interval for a binomial proportion; in Figure~\ref{fig4} the jumps in
its $p$-value function are shown as vertical lines, in order to make the
consequences of the non-nestedness easier to spot.

\begin{figure}

\includegraphics{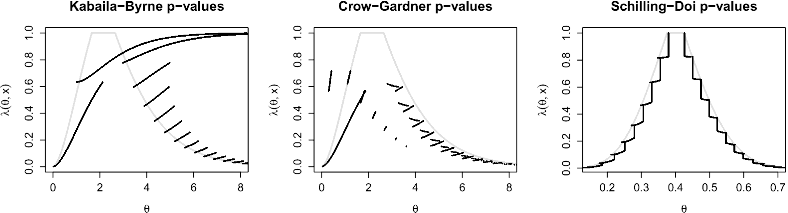}

\caption{Comparison between $p$-values corresponding to fiducial
intervals (grey) and $p$-values corresponding to some minimization-based
intervals (black). The $p$-values of the tests corresponding to the
Kabaila--Byrne and Crow--Gardner intervals are shown for $x=2$ being an
observation from a Poisson distribution with null mean $\theta$, and
the $p$-values of the test corresponding to the Schilling--Doi interval
are shown for $x=8$ being an observation from a null $\operatorname{Bin}(20,\theta
)$-distribution.}\label{fig4}
\end{figure}

\section{Continuity and bimonotonicity}\label{strictly}

For $\Theta\subseteq\mathbb{R}$, we say that a function $f:\Theta\to
\mathbb{R}$ is strictly bimonotone on $\Theta$ if there exist $\theta
_0,\theta_1\in\Theta$ such that $f$ is strictly increasing on $(\inf
\Theta,\theta_0)$, constant on $(\theta_0,\theta_1)$ and strictly
decreasing on $(\theta_1,\sup\Theta)$.

As have been argued for example, by  Hirji \cite{hirji06} and Vos and Hudson \cite{vh1},
this type of bimonotonicity is a highly desirable property of $p$-values
when viewed as a function of $\theta$. Ideally $\lambda(\theta,x)$
should increase monotonically from 0 to 1 and then decreases
monotonically to 0, just like the $p$-values of the tests corresponding
to fiducial intervals do in Figure~\ref{fig1}. One reason that this
property is desirable is the following result.

\begin{proposition}\label{propbim}
The bounds of a confidence interval are discontinuous in $\alpha$ if
their corresponding $p$-value function is not strictly bimonotone in
$\theta$.
\end{proposition}
\begin{pf}
Assume without loss of generality that there exist $\theta_0<\theta
_1<\inf\{\theta: \lambda(\theta,x)=1\}$ such that $\lambda(\theta,x)$
is increasing in $\theta$ in the interval $(\inf\Theta,\theta_0)$ and
decreasing or constant in the interval $(\theta_0,\theta_1)$. Let
$\alpha_0=\lambda(\theta_0,x)$. Then $\theta_1=\inf\{\theta>\theta_0:
\lambda(\theta,x)>\alpha_0\}$. Thus $L_{\alpha_0}(x)=\theta_0$ but for
all $\varepsilon>0$, $L_{\alpha+\varepsilon}(x)\geq\theta_1$, meaning that
$L_\alpha(x)$ has a jump of length $\theta_1-\theta_0>0$ at $\alpha
=\alpha_0$. An analogous argument holds for the upper bound.
\end{pf}

Hirji \cite{hirji06} mentions that $p$-value functions of strictly two-sided
tests need not be bimonotone, whereas Vos and Hudson \cite{vh1} showed by example
that the Blaker test for a binomial proportion lacks bimonotonicity.
Upon closer inspection of Figures~\ref{fig1} and \ref{fig2}, it can be
seen that all the strictly two-sided tests considered here suffer from
this problem.

Next, we give a condition under which the $p$-value function of a
strictly two-sided test is strictly bimonotone for fixed $x$, the
derivation of which is given in Appendix~\ref{proofs}. The
bimonotonicity condition requires the following additional assumptions,
which are satisfied by the binomial, negative binomial and Poisson
distributions.
\begin{longlist}[A4.]
\item[A4.] For\vspace*{1pt} $x\in\mathcal{X}\setminus\sup\mathcal{X}$, $\lim_{\theta
\rightarrow\inf\Theta}\Prob_\theta(X\leq x)=1$ and $\lim_{\theta
\rightarrow\sup\Theta}\Prob_\theta(X\leq x)=0$.
\item[A5.] For $k_1,k_2\in\mathcal{X}$ such that $k_1\geq k_2+2$, $\sum_{k\geq k_1 }\Prob_\theta(X=k)+\sum_{k\leq k_2 }\Prob_\theta(X=k)$ has
a unique minimum in the interior of $\Theta$.
\end{longlist}


\begin{proposition}\label{proposition1}
Under the assumptions and notation of Proposition~\ref{thm3}, assume
that $P_\theta$ satisfies conditions \textup{A4} and \textup{A5}. Let $\theta_r(\theta
_0,x)$ be the solution to
%
\begin{equation}
\label{Heq} \sum_{k=k_2(\theta_0,x)+1}^{k_1(\theta_0,x)-1}
\frac{d}{d\theta}\Prob _\theta(X=k)=0
\end{equation}
in the interior of $\Theta$. Then
\begin{longlist}[(a)]
\item[(a)] $\lambda(\theta,x)$ is strictly bimonotone in $\theta$ for
any fixed $x\in\mathcal{X}\setminus\sup\mathcal{X}$.
\item[(b)] The bounds of $I_\alpha(x)$ are continuous in $\alpha$,
\end{longlist}
if and only if there does not exist $(\theta_0,x)$ such that either
%
\begin{eqnarray}
&&\theta_0<\inf\bigl\{\theta: \lambda(\theta,x)=1\bigr\}\quad
\mbox{and}\quad \theta _0<\theta_r(\theta_0,x),\quad
\mbox{or}
\nonumber
\\[-8pt]
\label{bimcond}
\\[-8pt]
\nonumber
&&\theta_0>\sup\bigl\{\theta: \lambda(\theta,x)=1\bigr\}
\quad\mbox{and}\quad \theta _0>\theta_r(\theta_0,x).
\end{eqnarray}
\end{proposition}

For any given $P_\theta$, we can evaluate numerically whether the
bimonotonicity condition \eqref{bimcond} is violated for a pair $(\theta
_0,x)$. We have not been able to find a strictly two-sided test that
passes \eqref{bimcond} for any $x$. Proposition~\ref{proposition1} is
illustrated in the Poisson and binomial settings in Figures~\ref{fig1}--\ref{fig2}. When $x=2$ from a Poisson random variable has been
observed, the $p$-value functions of the Sterne and Blaker tests are
non-bimonotone for the first time when $\theta=3$. For the likelihood
ratio test, the first occurrence is at $\theta=1$ and for the score
test the first occurrence is at $\theta=\sqrt{12}$.

A consequence of $\lambda(\theta,x)$ lacking bimonotonicity is that the
confidence ``interval'' $\{\theta: \lambda(\theta,x)>\alpha\}$ may
contain holes, and therefore not be an interval at all. The common
remedy for this is to redefine the intervals as the convex hull of $\{
\theta: \lambda(\theta,x)>\alpha\}$, as we did in Definition~\ref{def3}. This does not change the infimum or supremum of the set, and
does therefore not affect nestedness or continuity of the bounds.
{Similarly, Fay \cite{fay10a} proposed handling the problem of
non-bimonotone $p$-value functions by redefining the $p$-values using the
convex hull of $\{\theta: \lambda(\theta,x)>\alpha\}$. The redefined
$p$-values are constant where they previously were non-monotone. By
Proposition~\ref{propbim}, the bounds of the corresponding intervals
are however still discontinuous in $\alpha$.}


For the binomial and negative binomial distributions, the left-hand
side of (\ref{Heq}) is a polynomial of order $k_1(\theta_0,x)-k_2(\theta
_0,x)-1$. For the Poisson distribution, it is straightforward to find a
general solution to (\ref{Heq}), which yields the following
proposition, the proof of which is omitted.

\begin{proposition}\label{lemmA1}
For $X\sim \operatorname{Poisson}(\theta)$, the $p$-value function $\lambda(\theta,x)$
belonging to a strictly two-sided test is bimonotone in $\theta$ if and
only if there does not exist $(\theta,x)$ such that either
\begin{itemize}
\item$\theta<\inf\{\theta: \lambda(\theta,x)=1\}$ and $\theta<
(\frac{(k_1(\theta,x)-1)!}{k_2(\theta,x)!} )^{1/(k_1(\theta
,x)-k_2(\theta,x)-1)}$,\vspace*{1.5pt} or
\item$\theta>\sup\{\theta: \lambda(\theta,x)=1\}$ and $\theta>
(\frac{(k_1(\theta,x)-1)!}{k_2(\theta,x)!} )^{1/(k_1(\theta
,x)-k_2(\theta,x)-1)}$.
\end{itemize}
\end{proposition}

Note that if we let $n=k_1(\theta,x)-k_2(\theta,x)-1$ then
\[
\biggl(\frac{(k_1(\theta,x)-1)!}{k_2(\theta,x)!} \biggr)^{1/(k_1(\theta
,x)-k_2(\theta,x)-1)}= \Biggl(\prod
_{k=k_2(\theta,x)+1}^{k_1(\theta
,x)-1}k \Biggr)^{1/n},
\]
the geometric mean of $\mathcal{A}_{\theta,x}^c$.

\section{Some results for fiducial intervals}\label{optimality}
\subsection{Fiducial intervals are strictly nested and have continuous bounds}
The test corresponding to the fiducial intervals is not strictly
two-sided. Its $p$-values are defined by \eqref{fidtest}.
The following proposition, the proof of which can be found in Appendix~\ref{proofs}, states that fiducial intervals do not suffer from the
problems associated with strictly two-sided intervals.

\begin{proposition}\label{fidpropr}
Under \textup{A1}--\textup{A3} and \textup{A4}, fiducial intervals are strictly nested.
Moreover, for any $x\in\mathcal{X}$ the bounds of the interval are
continuous in $\alpha$ and $\lambda_f(\theta,x)$ is continuous in
$\theta$.
\end{proposition}


\subsection{Optimality results}
For a binomial proportion, Wang \cite{wang06} presented results claiming
that under certain conditions on $\alpha$ and $n$ the fiducial
Clopper--Pearson interval is the shortest interval in the class of
exact confidence intervals with monotone bounds. A counterexample to
the optimality result of \cite{wang06} is the strictly two-sided
Blaker interval \cite{bl1}, which always is contained in the
Clopper--Pearson interval. Among equal-tailed intervals however,
fiducial intervals posses length optimality properties. We expect that
this is known, but have not been able to find such results in the
literature, for which reason we briefly cover length optimality below.

Our main tool for showing length optimality is a theorem due to \cite
{bolshev65}. Under assumptions A1, A2 and A3, consider the class
$\mathcal{M}_{L,\alpha}$ of one-sided $1-\alpha$ confidence bounds
$(L_\alpha(x),\infty)\cap\Theta$ for $\theta\in\Theta$ based on an
observation $x$ of $X\sim P_\theta$ satisfying the following three criteria:
\begin{longlist}[C3.]
\item[C1.] $L_\alpha(x)\leq L_\alpha(x+1)$,
\item[C2.] $\inf_{\theta\in\Theta}\Prob_\theta(L_\alpha(x)\leq\theta
)\geq1-\alpha$,
\item[C3.] $L_\alpha(x)$ only depends on $x$, $\alpha$ and $P_\theta$.
\end{longlist}

Criterion C3 rules out randomized bounds, which can be shorter while
maintaining exact coverage, but rely on conditioning on information not
contained in the sufficient statistic; see, for example, \cite{thu3}.
C3 is implicit in Bolshev's paper; we have added it here for clarity.
$\mathcal{M}_{L,\alpha}$ is the class of monotone exact lower
confidence bounds. We call an interval (or a bound) $I_\alpha(x)$ in a
class of intervals $\mathcal{K}$ the smallest interval in $\mathcal{K}$
if, for any other interval $I^*_\alpha(x)\in\mathcal{K}$, $I_\alpha
(x)\setminus I^*_\alpha(x)=\varnothing$. For the $\mathcal{M}_{L,\alpha
}$ class,  Bolshev \cite{bolshev65} proved that the one-sided lower fiducial
bound is the smallest bound in $\mathcal{M}_{L,\alpha}$. Under
analogous conditions, the upper fiducial bound is similarly the
smallest bound in the set $\mathcal{M}_{U,\alpha}$ of exact monotone
upper confidence bounds.

The extension of Bolshev's theorem to two-sided confidence intervals is
straightforward and does not require the additional conditions that
Wang
\cite{wang06} used in the binomial setting. Consider the class
$\mathcal{M}_\alpha$ of exact equal-tailed confidence intervals
$(L_{\alpha/2}(x),U_{\alpha/2}(x))$ for $\theta$ based on an
observation $x$ of $X\sim P_\theta$ satisfying
\begin{longlist}
\item[D1.] $L_{\alpha/2}(x)\leq L_{\alpha/2}(x+1)~$ and $~U_{\alpha
/2}(x)\leq U_{\alpha/2}(x+1)$,
\item[D2.] $\inf_{\theta\in\Theta}\Prob_\theta(L_{\alpha/2}(x)\leq\theta
)\geq1-\alpha/2~$ and $~\inf_{\theta\in\Theta}\Prob_\theta(U_{\alpha
/2}(x)\geq\theta)\geq1-\alpha/2$,
\item[D3.] $(L_{\alpha/2}(x),U_{\alpha/2}(x))$ only depends on $x$,
$\alpha$ and $P_\theta$.
\end{longlist}
Note that if an interval belongs to $\mathcal{M}_\alpha$ then it is the
intersection of a bound in $\mathcal{M}_{L,\alpha/2}$ and a bound in
$\mathcal{M}_{U,\alpha/2}$.

\begin{proposition}\label{thm4}
The fiducial interval is the smallest interval in $\mathcal{M}_\alpha$.
\end{proposition}
\begin{pf}
Let $I_\alpha(x)=(L_{\alpha/2}(x),U_{\alpha/2}(x))$ denote the fiducial
interval and assume that there is an interval $I^*_\alpha
(x)=(L^*_{\alpha/2}(x),U^*_{\alpha/2}(x))$ in $\mathcal{M}_\alpha$ such
that $I_\alpha(x)\setminus I^*_\alpha(x)\neq\varnothing$. Then
$L^*_{\alpha/2}(x)>L_{\alpha/2}(x)$ or $U^*_{\alpha/2}(x)<U_{\alpha
/2}(x)$. Consequently, at least one of the one-sided bounds
$(L^*_{\alpha/2}(x),\infty)\cap\Theta$ or $(-\infty,U^*_{\alpha
/2}(x))\cap\Theta$ is smaller than the corresponding fiducial bound. By
Bolshev's theorem, this means that $I^*_\alpha(x)$ is not in $\mathcal
{M}_\alpha$, which is a contradiction.
\end{pf}

Similar results can be obtained for intervals with fixed but unequal
tails, in a completely analogue manner.

Finally, the fact that the fiducial interval is the smallest interval
in $\mathcal{M}_\alpha$ leads to the following proposition, in which
the smallness is expressed in the more familiar terms of the interval
length $U_{\alpha/2}(x)-L_{\alpha/2}(x)$.

\begin{proposition}\label{cor1}
Among the intervals in $\mathcal{M}_\alpha$, the fiducial interval
minimizes the expected length for all $\theta\in\Theta$ as well as the
length for all $x\in\mathcal{X}$.
\end{proposition}
\begin{pf}
For an interval $(L^*_{\alpha/2}(x),U^*_{\alpha/2}(x))\in\mathcal
{M}_\alpha$ to have shorter length than the fiducial interval
$(L_{\alpha/2}(x),U_{\alpha/2}(x))$ it must hold that $L^*_{\alpha
/2}(x)>L_{\alpha/2}(x)$ or $U^*_{\alpha/2}(x)<U_{\alpha/2}(x)$. By\vspace*{1pt}
Proposition~\ref{thm4} neither condition can be fulfilled. Since the
fiducial interval therefore minimizes the length for each $x$, it also
minimizes the expected length $\sum_k\Prob_\theta(X=k)(U_{\alpha
/2}(k)-L_{\alpha/2}(k))$.
\end{pf}

A consequence of Proposition~\ref{cor1} is that, in the class of
equal-tailed two-sided tests of $\theta=\theta_0$, the test that
corresponds to the fiducial interval is admissible in the sense of
Cohen and Strawderman\vspace*{-6pt} \cite{cohen}.

\section{Conclusion}\label{discussion}
There exist a large number of methods for obtaining exact confidence
intervals that are shorter than the equal-tailed fiducial intervals.
The use of such an interval comes at the cost of losing control over
the balance between the coverage levels of the corresponding lower and
upper confidence bounds. In many situations it is preferable to use an
equal-tailed interval, in order to guard equally against overestimation
and underestimation and not to bias the inference in some direction.
The case for equal-tailed intervals is further strengthened by the fact
that strictly two-sided confidence intervals lack strict nestedness.
This causes difficulties with the interpretation of the intervals: what
does it mean that, for a particular $x$, the 92\% interval equals the
95\% interval? Which confidence level should be reported for such an
interval? More seriously, we have also seen that such intervals may
yield highly disparate conclusions for two indistinguishable models
$P_\theta$ and $P_{\theta+\varepsilon}$. From a hypothesis testing
perspective, this occurs when the null hypothesis $\theta_0$ is changed
slightly. From a confidence interval perspective, it can occur for
small changes in $\alpha$, since the bounds of strictly two-sided
intervals typically are discontinuous in $\alpha$. These problems have
been pointed out for specific intervals in the past \cite{bl1,vh1}. We
have shown that they in fact are inherent to strictly two-sided
confidence intervals.

The problems discussed in this paper arise also for strictly two-sided
methods for discrete distributions not covered by Definition~\ref
{def2}. Examples include the hypergeometric distribution and the joint
distribution of two binomial proportions. We have restricted our
attention to the class of distributions given by Definition~\ref{def2}
in order to keep the proofs reasonably short.

Strictly two-sided and equal-tailed confidence intervals are the most
commonly used types of two-sided confidence intervals. We have seen
that strictly-two sided intervals lack strict nestedness and that an
extension of Bolshev's theorem shows that the standard fiducial
intervals are the shortest equal-tailed exact intervals. While fiducial
intervals have been criticized for being overly conservative and too
wide \cite{bcd1,agresti2,byrne05}, the conclusion of this paper is
that they for practical purposes in fact are the optimal strictly
nested intervals.\vspace*{-3pt}

\begin{appendix}
\section{Strictly two-sided tests}\label{app1}
\subsection{The likelihood ratio and Sterne tests}
Let $L(\theta,x)=\Prob_\theta(X=x)$ be the likelihood function of
$P_\theta$. The likelihood ratio statistic is
\[
T_{\mathrm{LR}}(\theta_0,x)=\frac{\sup_{\theta\in\Theta}L(\theta,x)}{L(\theta_0,x)}
\]
and the Sterne statistic is
\[
T_{\mathrm{St}}(\theta_0,x)=1/L(\theta_0,x).
\]
Both these statistics are minimized when $\theta_0$ is the maximum
likelihood estimator of $\theta$ given~$x$. Thus for B1 to be satisfied
it suffices that the maximum likelihood estimator of $\theta$ is
well-defined and strictly monotone in $x$. B2 is satisfied when there
exists a $\theta$ such that $L(\theta,x)$ is an asymmetric function of
$x$. By definition, B3 is satisfied if there exists an $x_0$ such that
$L(\theta,x)$ is increasing when $x<x_0$ and decreasing when $x>x_0$.
This is guaranteed if $P_\theta$ has a monotone likelihood ratio.

The binomial, negative binomial and Poisson distributions all have
well-defined and strictly monotone maximum likelihood estimators and
monotone likelihood ratios. Moreover, their probability functions are
in general asymmetric in $x$. The likelihood ratio and Sterne tests
therefore satisfy conditions B1--B3 for these models.

\subsection{The score test}
Let $U(\theta,x)=\frac{\partial}{\partial\theta}\ln L(\theta,x)$ and
let $I(\theta)$ be the Fisher information of $P_\theta$. The score test
statistic is
\[
T_{Sc}(\theta_0,x)=\frac{(U(\theta_0,x))^2}{I(\theta_0)}.
\]

If the maximum likelihood estimator of $\theta$ exists and is unique,
then B1 is satisfied, with $\theta_x$ being the maximum likelihood
estimator of $\theta$ given $x$. B2 is satisfied if the distribution of
$U(\theta,x)^2$ is asymmetric for some $\theta$. B3 is satisfied if
there exists an $x_0$ such that $U(\theta,x)^2$ is decreasing when
$x<x_0$ and increasing when $x>x_0$.

If $P_\theta$ is a regular exponential family with natural parameter
$\theta$, then $U(\theta,x)=x-\mbox{E}_\theta(X)$ and $I(\theta)=\operatorname{Var}_\theta(X)$. B2 is satisfied if the distribution of $X^2$ is
asymmetric for some $\theta$ and B3 is satisfied since $(x-\mbox
{E}_\theta(X))^2$ is convex in $x$. B1--B3 are therefore satisfied for
the binomial, negative binomial and Poisson distributions, using the
natural parametrizations. These conditions are also satisfied for the
most commonly used alternative parametrizations.

\subsection{The Blaker test}
The Blaker statistic is
\[
T_B(\theta_0,x)\propto \bigl(\lambda_T(
\theta_0,x) \bigr)^{-1},
\]
where $\lambda_T(\theta_0,x)$ is the $p$-value of a test with a rejection
region that is the union of the rejection regions of two one-sided
level $\alpha/2$ tests. The properties of $T(\theta_0,x)$ therefore
depend on the choice of $\lambda_T(\theta_0,x)$. A typical choice is
the fiducial $p$-value \eqref{fidtest}.

Under\vspace*{1pt} A3 and A4, for any $x\in\mathcal{X}$ there exist $\theta_x$ such
that $\sum_{k\leq x-1}\Prob_{\theta_x}(X=k)<1/2$ and $\sum_{k\geq
x+1}\Prob_{\theta_x}(X=k)<1/2$. Then\vspace*{1pt} we have $\sum_{k\leq x}\Prob
_{\theta_x}(X=k)\geq1/2$ and $\sum_{k\geq x}\Prob_{\theta_x}(X=k)\geq
1/2$, so that $\lambda_f(\theta,x)=1$. Let $\Theta_1(x)$ denote the set
of such $\theta_x$.

Now,\vspace*{1.5pt} let $y=x-1$. Then $\sum_{k\leq y-1}\Prob_{\theta_x}(X=k)<1/2$ but
$\sum_{k\geq y+1}\Prob_{\theta_x}(X=k)=\sum_{k\geq x}\Prob_{\theta
_x}(X=k)\geq1/2$, so if $\theta_x\in\Theta_1(x)$ then $\theta_x\notin
\Theta_1(y)$. Similarly,\vspace*{1.5pt} if we let $y=x+1$, $\sum_{k\leq y-1}\Prob
_{\theta_x}(X=k)=\sum_{k\leq x}\Prob_{\theta_x}(X=k)\geq1/2$, so if
$\theta_x\in\Theta_1(x)$ then $\theta_x\notin\Theta_1(y)$. Thus,\vspace*{1.5pt} A3 and
A4 are sufficient for B1 to hold for the Blaker statistic based on the
fiducial $p$-value.

B2 holds if the distribution of $\lambda_T(\theta,x)$ is asymmetric in
$x$ for some $\theta$. For $\lambda_f(\theta,x)$ this holds if $\Prob
_\theta(X= x)$ is asymmetric as a function of $x$ for some $\theta$.

Finally, B3 is satisfied since the monotonicity of $\Prob_\theta(X\leq
x)$ in $x$ implies that $\lambda_T(\theta,x)$ is a bimonotone function
of $x$. B1--B3 are therefore satisfied for the binomial, negative
binomial and Poisson distributions.

\section{Proofs}\label{proofs}
\subsection{Proof of Proposition \texorpdfstring{\protect\ref{expfamprop}}{1}}
If $P_\theta$ is a discrete one-parameter exponential family in natural
form, for $(\theta,x)\in\Theta\times\mathcal{X}$ its probability
function can be written as
%
\begin{equation}
\label{expfam} p_\theta(x)=\Prob_\theta(X=x)=\exp\bigl\{\theta
T(x)-K(\theta)\bigr\}h(x),
\end{equation}
where $T: \mathcal{X}\mapsto\mathbb{R}$ is a function that does not
depend on $\theta$ and $K: \Theta\mapsto\mathbb{R}$ is infinitely often
differentiable in $\Theta$ since $P_\theta$ is regular \cite{liese}, Theorem~1.17. $\theta$, $T(x)$ and $K(\theta)$ are all finite for
$(\theta,x)\in\Theta\times\mathcal{X}$, and thus \eqref{expfam} is
strictly positive when $(\theta,x)\in\Theta\times\mathcal{X}$, yielding
A1. Moreover, A3 follows from the fact that when $x$ is fixed \eqref
{expfam} is differentiable in $\theta$ since $\theta T(x)$ and $K(\theta
)$ are infinitely differentiable.

To see that an increasing likelihood ratio implies A2, let $\ell
(x)=p_{\theta_2}(x)/p_{\theta_1}(x)$ for $\theta_1<\theta_2$ in~$\Theta
$. The likelihood ratio $\ell(x)$ is increasing in $x$. Let
\[
F_\theta(x)=\Prob_\theta(X\leq x)=\sum
_{k\leq x}\Prob_\theta(X=k)
\]
and
\[
G_\theta(x)=\Prob_\theta(X> x)=\sum
_{k> x}\Prob_\theta(X=k).
\]
We consider the cases when $\ell(x)\leq1$ and $\ell(x)\geq1$ separately.

If $\ell(t)\leq1$ then for $s<t$, $F_{\theta_2}(s)\leq F_{\theta
_1}(s)$ since $p_{\theta_2}(x)\leq p_{\theta_1}(x)$ for all $x<t$.

If $\ell(t)\geq1$ then $p_{\theta_2}(x)\geq p_{\theta_1}(x)$ when
$x>t$ and for $s>t$, $G_{\theta_2}(s)\geq G_{\theta_1}(s)$. Since
$F_\theta(x)=1-G_\theta(x)$, it follows that $F_{\theta_2}(s)\leq
F_{\theta_1}(s)$.

\subsection{Proof of Proposition \texorpdfstring{\protect\ref{B3prop}}{2}}
First, assume that $k\geq x_\theta$. Then by B3 $T(\theta,\cdot)$ is
increasing at $k$. There are two possible scenarios:
\begin{longlist}[(ii)]
\item[(i)] $k\geq k_1(\theta,x)$: By definition, $T(\theta,k_1(\theta
,x))\geq T(\theta,x)$. Since $T(\theta,\cdot)$ is increasing for $x\geq
x_\theta$ it follows that $T(\theta,k)\geq T(\theta,k_1(\theta,x))\geq
T(\theta,x)$, meaning that $k\in\mathcal{A}_{\theta,x}$.
\item[(ii)] $k< k_1(\theta,x)$: it follows from the definition of
$k_1(\theta,x)$ that $T(\theta,k)<T(\theta,x)$, so $k\notin\mathcal
{A}_{\theta,x}$.
\end{longlist}
In summary, if $k\geq x_\theta$ then $k\in\mathcal{A}_{\theta,x}$ if
and only if $k\geq k_1(\theta,x)$. An analogous argument shows that if
$k\leq x_\theta$ then $k\in\mathcal{A}_{\theta,x}$ if and only if
$k\leq k_2(\theta,x)$, and the first part of the proposition follows.

To see that at least one of $k_1(\theta,x)$ and $k_2(\theta,x)$ is
non-constant in $\theta$, note that by B1, for any pair $(x,y)\in
\mathcal{X}^2$ there exist $(\theta_x,\theta_y)\in\Theta^2$ such that
$x\notin\mathcal{A}_{\theta_x,y}$ but $x\in\mathcal{A}_{\theta_y,y}$.
The set $\mathcal{A}_{\theta,x}$ is therefore not constant in $\theta$,
and thus at least one of $k_1(\theta,x)$ and $k_2(\theta,x)$ must be
non-constant in~$\theta$.


\subsection{Proof of Proposition \texorpdfstring{\protect\ref{thm3}}{3}}
(a) and (b) were proved in Section~\ref{lackofstrict}. We will now
prove (c). Let $L_{\alpha}(x)$ and $U_{\alpha}(x)$ denote the lower and
upper bounds of the interval. We will show that for any $x\in\mathcal
{X}$ there exists an $\alpha_0\in(0,1)$ such that $L_\alpha(x)$ and
$U_\alpha(x)$ simultaneously are constant in a neighbourhood of $\alpha
_0$, so that the confidence interval is not strictly nested.

We introduce the mutually disjoint sets
%
\begin{eqnarray}
\Theta_1(x) &:=& \bigl\{\theta: \lambda(\theta,x)=1\bigr
\},
\nonumber
\\
\label{Thetas}
\Theta_\ell(x)&:=&\bigl\{\theta: \theta\leq\inf
\Theta_1(x)\bigr\} \quad\mbox{and}
\\
\Theta_u(x) &:=& \bigl\{\theta: \theta\geq\sup\Theta_1(x)
\bigr\},
\nonumber
\end{eqnarray}
which are such that $\Theta_\ell(x)\cup\Theta_1(x)\cup\Theta_u(x)=\Theta$. We also define
\[
\Theta_\alpha(x):=\bigl\{\theta: \lambda(\theta,x)> \alpha\bigr\}.
\]

By condition B1, given $x\in\mathcal{X}$ there exists $\theta_x\in\Theta
$ such that
\[
\lambda(\theta_x,x)=\Prob_{\theta_x}\bigl(T(
\theta_x,X)\geq T(\theta _x,x)\bigr)=1-
\Prob_{\theta_x}\bigl(T(\theta_x,X)< T(\theta_x,x)
\bigr)=1,
\]
so $\Theta_1(x)$ is non-empty. 

First, we investigate the behaviour of the bounds when either $\Theta
_\ell(x)$ or $\Theta_u(x)$ is empty. Let $\bar{\Theta}_\alpha(x)$ be
the closure of $\Theta_\alpha(x)$. Since $\Theta_1(x)\subseteq\bar
{\Theta}_\alpha(x)$ for all $\alpha\in(0,1)$,
\[
\sup\Theta_1(x)\in\bar{\Theta}_\alpha(x) \quad\mbox{and}\quad \inf
\Theta _1(x)\in\bar{\Theta}_\alpha(x).
\]
If $\inf\Theta_1(x)=\inf\Theta$, then $\Theta_\ell(x)=\{\theta: \theta
<\inf\Theta\}=\varnothing$. Then
\[
L_\alpha(x)=\inf\Theta_\alpha(x)\leq\inf\Theta_1(x)=
\inf\Theta\leq\inf \Theta_\alpha(x),
\]
so $L_\alpha(x)=\inf\Theta$ for all $\alpha\in(0,1)$. Similarly, if
$\sup\Theta_1(x)=\sup\Theta$ then $\Theta_u(x)=\{\theta: \theta>\sup
\Theta\}=\varnothing$, and
\[
U_\alpha(x)=\sup\Theta_\alpha(x)\geq\sup\Theta_1(x)=
\sup\Theta\geq\sup \Theta_\alpha(x),
\]
so $U_\alpha(x)=\sup\Theta$ for all $\alpha\in(0,1)$. Thus, when $\Theta
_\ell(x)$ is empty $L_\alpha(x)$ is constant and when $\Theta_u(x)$ is
empty $U_\alpha(x)$ is constant. In this case, whether or not the
interval is strictly nested therefore depends on whether there exists
an $\alpha\in(0,1)$ such that the other bound is constant in a
neighbourhood of $\alpha$. We will therefore without loss of generality
assume that neither $\Theta_\ell(x)$ nor $\Theta_u(x)$ are empty.


Let
\[
\alpha_\ell=\liminf_{\theta\to\inf\Theta_1(x)}\lambda(\theta,x)
\quad\mbox{and}\quad \alpha_u=\liminf_{\theta\to\sup\Theta_1(x)}\lambda (
\theta,x).
\]
Since $\mathcal{A}_{\theta,x}\neq\mathcal{X}$ for $\theta<\inf\Theta
_1(x)$, by A1 $\alpha_\ell<1$, and similarly $\alpha_u<1$. Thus, a
point is added to or removed from $\mathcal{A}_{\theta,x}$ at $\theta
=\inf\Theta_1(x)$ and $\theta=\sup\Theta_1(x)$, and by Lemma~\ref{B4}
the $p$-value function $\lambda(\theta,x)$ must have jumps at $\theta=\inf
\Theta_1(x)$ and at $\theta=\sup\Theta_1(x)$. Then for $\alpha_0=\max
(\alpha_\ell,\alpha_u)$, there is an $\alpha_1\in(\alpha_0,1)$ for
which there exists $\delta>0$ such that
\[
\bigl\{\theta\in\Theta: \lambda(\theta,x)\in(\alpha_1-\delta,
\alpha_1+\delta )\bigr\}=\varnothing.
\]
Thus both the upper and the lower bound of $I_\alpha(x)$ are constant
in a neighbourhood of $\alpha=\alpha_1$, and the interval is not
strictly nested.


\subsection{An auxiliary lemma}
The following auxiliary lemma will be used in the proof of Proposition~\ref{smallprop}.

\begin{lemma}\label{lemmat}
With $\mathcal{X}$ as in Proposition~\ref{smallprop}, for any $y,x\in
\mathcal{X}$ such that $0\leq y<x$, let
%
\begin{equation}
\label{lemma1}\theta_{y,x}=\inf\bigl\{\theta: T(\theta ,x)\leq T(
\theta,y)\bigr\}
\end{equation}
and define $\theta_{x,x}:=\inf\Theta_1(x)$. Under \textup{A2}$^+$ and \textup{B3}$^+$,
%
\begin{equation}
\label{lemma4}\theta_{0,x}\leq\theta_{1,x}\leq\cdots\leq
\theta_{x-1,x}\leq\theta_{x,x}.
\end{equation}
Moreover,
%
\begin{equation}
\label{lemma3}\mathcal{A}_{\theta,x}=\{k\in\mathcal{X}: k\geq x\}\quad \mbox{if
and only if} \quad \theta\in(\inf\Theta,\theta_{0,x}),
\end{equation}
and
%
\begin{equation}
\label{lemma2}\mathcal{A}_{\theta,x}=\{k\in\mathcal {X}:k\leq y\}\cup\{k\in
\mathcal{X}: k\geq x\}\quad \mbox{if and only if}\quad \theta\in[\theta_{y,x},
\theta_{y+1,x}).
\end{equation}
\end{lemma}

\begin{pf}
First, we establish some facts about $\theta_{y,x}$ and the behaviour
of $T(\theta,\cdot)$ for such $\theta$. If $\theta\in\Theta_1(x)$, then
it follows from \eqref{Aset} that $T(\theta,x)\leq T(\theta,y)$. Thus,
by \eqref{lemma1} we have $\theta_{y,x}\leq\inf\Theta_1(x)$, so by
\eqref{Thetas}, $\theta_{y,x}\in\Theta_\ell(x)$.

With $x_\theta$ as defined in B3$^+$(i), $T(\theta,\cdot)$ is
increasing at $x$ if $x>x_\theta$. If $\theta\in\Theta_1(x)$ then
$x_\theta=x$. By B3$^+$(iii), $x_\theta$ is an increasing function of
$\theta$. Thus, if $\theta\leq\inf\Theta_1(x)$, i.e. $\theta\in\Theta
_\ell(x)$, we have $x>x_\theta$. Since $\theta_{y,x}\in\Theta_\ell(x)$,
$T(\theta_{y,x},\cdot)$ is increasing at $x$.

It now follows that for any $y<x$, $T(\theta,x)<T(\theta,y)$ can happen
only if $T(\theta,\cdot)$ is decreasing at $y$. Whenever $T(\theta
,x)<T(\theta,y)$ and $T(\theta,\cdot)$ is decreasing at $y$, we have
$T(\theta,x)<T(\theta,y)<T(\theta,y-1)$, and \eqref{lemma4} follows
since $\{\theta: T(\theta,x)\leq T(\theta,y)\}\subset\{\theta: T(\theta
,x)\leq T(\theta,y-1)\}$.

Let\vspace*{1pt} $x\geq1$ be fixed. If $\theta\leq\theta_{0,x}$ then $T(\theta
,x)<T(\theta,0)$ and \eqref{lemma4} ensures that $T(\theta,x)<T(\theta
,y)$ for all other $y<x$ as well. \eqref{lemma3} now follows from \eqref{Aset}.

Next, for some $y<x$, let $\theta\in\Theta_\ell(x)$ be such that $\theta
>\theta_{y,x}$. Under B3$^+$(ii) we have $\theta<\theta_x\in\Theta
_1(x)$, so $T(\cdot,x)$ is decreasing in $\Theta_\ell(x)$. However, if
$T(\theta,x)<T(\theta,y)$ then $y<x_\theta$, so $\theta\geq\sup\Theta
_1(y)$. Thus $\theta_{x,y}>\theta_y$, implying that $T(\cdot,y)$ is
increasing at $\theta_{y,x}$. Thus\vspace*{1pt} $T(\theta,x)<T(\theta,y)$ for all
$\theta>\theta_{y,x}$. Equations \eqref{lemma3} and \eqref{lemma2} now follow
from \eqref{Aset}.
\end{pf}

%

\subsection{Proof of Proposition \texorpdfstring{\protect\ref{smallprop}}{4}}
We start by showing (b) and finish by proving (a). The proof of (c) is
analogous to the proof of~(b), and is therefore omitted.

(b) We wish to find the largest $\alpha_L$ such that, for all $x\in
\mathcal{X}$, $L_\alpha(x)$ is strictly monotone in $\alpha$ when
$\alpha<\alpha_L$. For a given $x$, let $\alpha_L(x)$ be the largest
$\alpha$ such that $L_\alpha(x)$ is strictly monotone in $\alpha$ when
$\alpha<\alpha_L(x)$. Then $\alpha_L\leq\alpha_L(x)$ for all $x$, with
equality for some $x$. We therefore show the statement by showing that
$\alpha_L(x)=\inf_{\theta\in\{\theta: T(\theta,0)>T(\theta,x)\}}\lambda
(\theta,x)>0$.

As in the proof of Lemma~\ref{lemmat}, it suffices to study $\theta\in
\Theta_\ell(x)$, where $\Theta_\ell(x)$ is defined as in~\eqref{Thetas}.
\begin{longlist}[(iii)]
\item[(i)]  Let $x\geq1$ be fixed. If $\theta\leq\theta_{0,x}$, defined as in
\eqref{lemma1}, then by \eqref{lemma3}, $\mathcal{A}_{\theta,x}=\{
k:k\geq x\}$. Thus, by \eqref{Aset}, $\lambda(\theta,x)=\Prob_\theta
(X\geq x)$. The $p$-value function $\lambda(\cdot,x)$ is therefore
non-negative (by~A1), yielding (i).

\item[(ii)]  $\lambda(\cdot,x)$ is strictly increasing (by A2$^+$) and
continuous (by A3). We extend the $p$-value function by defining $\lambda
(\inf\Theta,x):=\lim_{\theta\searrow\inf\Theta} \lambda(\theta,x)$.
Then $\lambda(\theta,x)$ is a continuous strictly monotone bijection
from the compact set ${}[\inf\Theta,\theta_{0,x}]$ to the
compact set ${}[\lambda(\inf\Theta,x),\alpha_L(x)]$. It is
therefore a homeomorphism, and it follows that its inverse $L_\alpha
(x)$ is continuous and strictly monotone in $\alpha$, which yields (ii).

\item[(iii)] 
From Lemmas \ref{lemmat} and \ref{B4}, it follows that the first
discontinuity in $\lambda(\theta,x)$ occurs at $\theta_{0,x}$. From
Definition~\ref{def3}, it follows that there exists an $\varepsilon>0$
such that $L_\alpha(x)$ is constant in $(\alpha_L(x),\alpha
_L(x)+\varepsilon)$ if and only if $\lambda(\theta,x)>\alpha_L(x)$ for all
$\theta\in{}[\theta_{0,x},\inf\Theta_1(x)]$.
\end{longlist}

By \eqref{lemma2} for any $\theta\in{}[\theta_{0,x},\inf\Theta
_1(x)]$, there exists a $y<x$ such that $\theta\in{}[\theta
_{y,x},\theta_{y+1,x})$, so that
%
\begin{eqnarray}
\lambda(\theta,x)&=& \Prob_\theta(X\leq y)+\Prob_\theta(X\geq
x)>\varepsilon +\Prob_\theta(X\geq x)
\nonumber
\\[-8pt]
\label{epsiloneq}
\\[-8pt]
\nonumber
&>& \varepsilon+\Prob_{\theta_{0,x}}(X\geq x)=\varepsilon+
\alpha_L(x)>\alpha_L(x),
\end{eqnarray}
where the first inequality follows from A1 and the second inequality
follows from A2$^+$. (iii) now follows.

\begin{longlist}[(a)]
\item[(a)]  For any $x\in\mathcal{X}$, let $A_L(x)$ denote the set of $\alpha$
for which $L_\alpha(x)$ is locally constant in $\alpha$ and $A_U(x)$
denote the set of $\alpha$ for which $U_\alpha(x)$ is locally
constant in $\alpha$. By definition, the largest $\alpha$ for which
$I_\alpha(X)$ is strictly nested is
\[
\alpha_{\mathrm{nest}}=\inf \biggl\{\alpha: \exists\varepsilon>0 \mbox{ for which
} (\alpha,\alpha+\varepsilon)\subseteq\bigcup_{x\in\mathcal{X}}
\bigl( A_L(x)\cap A_U(x) \bigr) \biggr\}.
\]
Using Proposition~\ref{thm3}(c), $A_L(x)\cap A_U(x)$ has a connected
uncountable subset for all $x$, so $\alpha_{\mathrm{nest}}$ always exists. By
part (b) of Proposition~\ref{smallprop}, $\alpha_{\mathrm{nest}}\geq\alpha_L>0$.
\end{longlist}

\subsection{Proof of Proposition \texorpdfstring{\protect\ref{proposition1}}{6}}
By Proposition~\ref{thm3}(a) and A3, $\lambda(\theta,x)$ is a
piecewise continuous function. By Lemma~\ref{B4}, it is not continuous
at the boundaries of the set $\Theta_1(x)$. Hence $\lambda(\theta,x)$
can only be bimonotone if it is monotone whenever it is continuous.
Each of its continuous parts can be represented by equation (\ref
{lambdaeq}) with fixed $k_1$ and $k_2$. Such a part can be written as
%
\begin{equation}
\label{lambdaeq2} 1-\sum_{k\leq k_1-1 }\Prob_\theta(X=k)+
\sum_{k\leq k_2 }\Prob_\theta(X=k).
\end{equation}
By A2,\vspace*{1.5pt} $1-\sum_{k\leq k_1-1 }\Prob_\theta(X=k)$ is strictly increasing
and $\sum_{k\leq k_2 }\Prob_\theta(X=k)$ is strictly decreasing. By
condition A4 (\ref{lambdaeq2}) equals 1 at the boundaries of $\Theta$.
If it is not constant, it must therefore by condition A5 have a unique
minimum in the interior of $\Theta$. Rewriting the expression again, we have
\[
\mbox{(\ref{lambdaeq2})}=1-\sum_{ k_2+1\leq k\leq k_1-1 }\Prob_\theta(X=k),
\]
so that the minimum is given by the root $\theta_r$ of the equation
%
\begin{equation}
\label{root} \frac{d}{d\theta}\sum_{k=k_2+1}^{k_1-1}
\Prob_\theta(X=k)=\sum_{k=k_2+1}^{k_1-1}
\frac{d}{d\theta}\Prob_\theta(X=k)=0
\end{equation}
that is in the interior of $\Theta$. Next, we let $k_1$ and $k_2$ vary
as functions of $(\theta,x)$ and use $\theta_r(\theta,x)$ to denote the
solution of \eqref{root} with $k_1=k_1(\theta,x)$ and $k_2=k_2(\theta,x)$.

By Proposition~\ref{thm3}(a), $\lambda(\theta,x)$ has jumps
corresponding to changes in $k_1(\theta,x)$ or $k_2(\theta,x)$. $\lambda
(\theta,x)$ fails to be bimonotone if
\[
\bigl(k_1(\theta,x),~k_2(\theta,x) \bigr)=
\bigl(k_1\bigl(\theta_r(\theta ,x)+\varepsilon,x
\bigr),~k_2\bigl(\theta_r(\theta,x)+\varepsilon,x\bigr)
\bigr)\qquad \mbox{for some }\varepsilon>0,
\]
i.e. if it does not jump before the root $\theta_r(\theta,x)$ that
corresponds to $(k_1(\theta,x),k_2(\theta,x))$, since for $\lambda
'(\theta,x)=\frac{d}{d\theta}\lambda(\theta,x)$,
\[
\operatorname{sign} \bigl(\lambda'\bigl(\theta_r(\theta,x)-
\varepsilon\bigr) \bigr)\neq\operatorname{sign} \bigl(\lambda'\bigl(
\theta_r(\theta,x)+\varepsilon\bigr) \bigr).
\]

Assume that $\theta\in\Theta_u(x)$. Then $\lambda(\theta,x)$ should be
decreasing in $\theta$. If $\theta>\theta_r(\theta,x)$ then (\ref
{lambdaeq2}) with $k_1=k_1(\theta,x)$ and $k_2=k_2(\theta,x)$ is
increasing, so that $\lambda(\theta,x)$ is not bimonotone. If instead
we assume that $\theta\in\Theta_\ell(x)$ so that $\lambda(\theta,x)$ is
in its increasing part, we similarly get that $\lambda(\theta,x)$ is
not bimonotone if $\theta<\theta_r(\theta,x)$. This establishes (a).
Part (b) then follows from Proposition~\ref{propbim}.

\subsection{Proof of Proposition \texorpdfstring{\protect\ref{fidpropr}}{8}}
Let $\overline{\mathbb{R}}$ be the extended real line and $\overline
{\Theta}\subseteq\overline{\mathbb{R}}$ be the closure of $\Theta$. Let
$F(\theta,x)=\Prob_\theta(X\leq x)$. By~A1--A4, $F(\cdot,x)$ is a
continuous monotone bijection from $\overline{\Theta}$ to ${}[
0,1]$ for all $x\in\mathcal{X}\setminus\sup\mathcal{X}$. Since
$\overline{\Theta}$ and ${}[0,1 ]$ both are compact, it
follows that $F(\cdot,x)$ is a homeomorphism, which ensures that the
bounds given by (\ref{fidint}) are continuous in $\alpha$. The
monotonicity of $F_\theta(\cdot,x)$ ensures that both $F_\theta
^{-1}(\cdot,x)$ and the bounds are monotone, so that the interval is
strictly nested.

Finally, by condition A3, the $p$-value function (\ref{fidtest}) is
continuous in $\theta$ when $x\in\mathcal{X}$ is fixed.
\end{appendix}

\section*{Acknowledgments}
The authors wish to thank the Editor and the reviewers for comments
that helped improve the paper.






\printhistory
\end{document}